\documentclass[10pt]{amsart}
\usepackage{amscd,amssymb,amsthm,amsmath,tikz-cd,adjustbox}
\usepackage[all]{xy}
\begin{document}

\newtheorem{prop}{Proposition}[section]

\newtheorem{thm}[prop]{Theorem}
\newtheorem{lemma}[prop]{Lemma}
\newtheorem{cor}[prop]{Corollary}
\newtheorem{dfn}[prop]{Definition}
\newtheorem{pro}[prop]{Proposition}
\theoremstyle{definition}
\newtheorem{Question}[prop]{Question}
\newtheorem{Example}[prop]{Example}
\newtheorem{Examples}[prop]{Examples}
\newtheorem{Remark}[prop]{Remark}
\newcommand{\ann}{\mbox{ann}}
\newcommand{\depth}{\mbox{depth}}
\newcommand{\br}{{\bf r}}
\newcommand{\md}{\operatorname{mod}}
\newcommand{\finmd}{\operatorname{fin-mod}}
\newcommand{\add}{\operatorname{add}}
\newcommand{\End}{\operatorname{End}}
\newcommand{\Hom}{\operatorname{Hom}}
\renewcommand{\Im}{\mbox{Im}}
\newcommand{\pd}{\mbox{pd}}
\newcommand{\Ker}{\mbox{Ker}}
\newcommand{\Coker}{\mbox{Coker}}
\newcommand{\coh}{\operatorname{coh}}
\newcommand{\soc}{\mbox{soc}}
\newcommand{\Tor}{\mbox{Tor}}
\newcommand{\Spec}{\mbox{Spec}}
\newcommand{\tp}{\mbox{top}}
\renewcommand{\dim}{\mbox{dim}}
\newcommand{\rad}{\mbox{rad}}
\newcommand{\gldim}{\mbox{gl.dim}}
\newcommand{\nilmd}{\operatorname{nilmod}}
\newcommand{\Ext}{\operatorname{Ext}\nolimits}
\newcommand{\op}{^{\mbox{op}}}
\newcommand{\pr}{^{\prime}}
\newcommand{\f}{\operatorname{fin}}
\newcommand{\Supp}{\operatorname{Supp}}
\newcommand{\semi}{\mathbin{\vcenter{\hbox{$\scriptscriptstyle|$}}
\;\!\!\!\times }}
\newcommand{\cx}{\operatorname{cx}\nolimits}
\renewcommand{\depth}{\operatorname{depth}\nolimits}
\newcommand{\dm}{\operatorname{dim}\nolimits}
\newcommand{\rk}{\operatorname{rk}\nolimits}
\newcommand{\mto}{\hookrightarrow} 
\newcommand{\da}{\downarrow} 
\newcommand{\str}{\stackrel}
\newcommand{\ra}{\rightarrow}
 \newcommand{\lra}{\longrightarrow}
\newcommand{\AR} {Auslander-Reiten }
\newcommand{\Tr}{\operatorname{Tr}\nolimits}
\newcommand{\cF}{\mathcal F}
\newcommand{\rank}{\operatorname{rank}\nolimits}

\title[Thick subcategories of the stable category]
{Thick subcategories of the stable category of modules over the exterior algebra.}

\author{Otto Kerner}\address{Mathematisches Institut
\\Heinrich-Heine-Universit\"at \\ 40225 D\"usseldorf\\ 
Germany} \email{kerner@math.uni-duesseldorf.de}\author{Dan Zacharia}
\address{Department of Mathematics\\Syracuse University\\Syracuse
13244\\USA} \email{zacharia@syr.edu}\thanks{This project was started
 while both authors were visiting the University of 
Bielefeld as part of the SFB's ``Topologische und spektrale 
Strukturen in der Darstellungstheorie" program. Most of the results in this paper were obtained in Bielefeld, and also during the authors' visit to the University of Kiel. We both thank Henning Krause and  
the SFB, and Rolf Farnsteiner for inviting us and for making our stays possible. } 

\dedicatory{Dedicated to Jos\'e Antonio de la Pe\~{n}a  for his 60th birthday}

\begin{abstract} 
We study thick subcategories defined by modules of complexity one in $\underline{\md}R$, where
$R$ is the exterior algebra in $n+1$ indeterminates.
\end{abstract}
\maketitle


Let $R$ be the exterior algebra in $n+1$ indeterminates, and let $\underline\md R$ denote the stable 
category of all the finitely generated graded $R$-modules. This partially semi-expository article is part of a project devoted to 
the study of the thick subcategories of $\underline\md R$. Note that using the Bernstein-Gelfand-Gelfand correspondence (\cite{BGG}), 
this is equivalent to studying the thick subcategories of $\mathcal D^b(\coh {\bf P}^n)$-the bounded derived category of coherent sheaves 
on the projective $n$-space. This lattice of thick subcategories has been studied before, see for instance the treatment in the $n=1$ in \cite{Kr} using non-crossing partitions (\cite{IT}), and also in \cite{GKR}. More generally the lattice of thick subcategories of the derived category of a noetherian ring has also been treated extensively in \cite{BIK, DS, Ho, N, T} to name just a few of the references. While we make use of some of the results in the literature mentioned above, our approach to this problem is different, in the sense that it relies on representation theoretic concepts and methods (Auslander-Reiten theory for instance). 

\medskip
\noindent We concentrate in this paper on the study of the thick subcategories of $\underline\md R$ generated 
by the graded modules of complexity one, which, while having a very nice description, turn out to be quite complicated. Eisenbud 
has proved in \cite{E}, that the complexity one graded modules are periodic up to a shift, and so, using the BGG correspondence (see also \cite{OSS}), they correspond to the bounded complexes of coherent sheaves of finite length. It is also easy to see that these modules 
are weakly Koszul in the sense of \cite{MVZ1}, see also \cite{HI}. 

\medskip
\noindent The article is organized as follows: In the first section 
we review the notation and the background needed. In particular for all non zero linear forms $\xi$ we consider the local linear 
modules $M_{\xi}=R/{\langle\xi\rangle}$ of complexity one. In the second section, we use \cite{E} to show that these are the building blocks of the modules 
of complexity one (Theorem \ref{modules of cx 1}), and we study the subcategories $\mathcal F(M_{\xi})$ consisting of those graded 
$R$-modules having a filtration with factors equal to $M_{\xi}$. We prove in Theorem \ref{nilmod} that these subcategories  are 
wild for $n>2$ (tame in the case when $n=2$), and they are all equivalent to the category of finite dimensional modules over the (commutative) power series ring $\Bbbk[[t_1,\cdots,t_n]]$ in $t_1,\ldots,t_n$. We also show that every minimal ideal closed thick category of $\underline\md R$  
containing a module of complexity one is generated by a module of the form $M_{\xi}$ as above.
In the third section, we give some criteria for deciding when modules of non maximal complexity have self-extensions. 
Finally, in the appendix we describe the thick subcategories in the projective line case.

\section{Introductory results and background}\label{intro}

Let $V$ be an $n+1$ dimensional vector space over an algebraically closed field $\Bbbk$ and 
let $R =R(V)= \bigwedge V$ be the corresponding exterior  algebra.  If $\{x_0,\ldots ,x_n\}$ is a basis 
of $V$, we also write $R= R(x_0,\ldots ,x_n)$ and
call $R$ the exterior algebra in $x_0,\ldots ,x_n$. It is well-known that $R$ 
is a graded $\Bbbk$-algebra by assigning each indeterminate $x_i$ degree one, and that as a $\Bbbk$-algebra, 
$R$ is generated by its degree zero and degree one (also called the non zero linear forms) parts.
 Let $\md R$ be the abelian category of finitely generated {\em graded} $R$-modules. 
For two graded $R$-modules $M$ and $N$, $\Hom_R(M,N)$ will {\it always} denote the space of degree 
zero homomorphisms from $M$ to $N$. Also, for two graded $R$-modules $M$ and $N$, $\Ext_R^i(M,N)$ 
will {\em always} denote the derived functors of the graded Hom between the two modules. If $M = \oplus_jM_j$
 is a graded $R$-module, then its {\it graded shift} 
is the graded module $M(i)$  with $M(i)_j=M_{i+j}$ for each $j\in\mathbb Z$.  By $\underline\md R$ 
we denote the stable category of finitely generated {\it graded} $R$-modules.  
The objects of $\underline\md R$ are the finitely generated graded $R$-modules, 
and $\underline\Hom_R(M,N)$ denotes the vector space $\Hom_R(M,N)/{\mathcal P(M,N)}$ 
where $\mathcal P(M,N)$ is the space of the degree zero homomorphisms from $M$ to $N$ 
factoring through a free $R$-module. For every finitely generated non projective graded 
$R$-modules $L$ and $M$ and homomorphism  $L\stackrel{u}\rightarrow M$ in 
$\md R$ we have a pushout diagram (\cite{H}) 
$$ 
\xymatrix
{ &0\ar[r]&L \ar^{u}[d]
\ar[r] &I(L)\ar[d]\ar[r]
&\Omega^{-1}L\ar@{=}[d] \ar[r] &0\\ &0 \ar[r] &M \ar^{v}[r]
&N\ar^{w}[r] &\Omega^{-1}L\ar[r]& 0} 
$$
\noindent where $\Omega$ (respectively $\Omega^{-1}$) denote the syzygy (cosyzygy) functors  
$$\Omega,\Omega^{-1}\colon\underline\md R\rightarrow\underline\md R$$ and $I(L)$ 
denotes the injective envelope of $L$. It is well-known that $\underline\md R$ 
is a triangulated category where the shift (or suspension) functor is given by the first cosyzygy, 
and every distinguished triangle is isomorphic in $\underline\md  R$ to a triangle of the form 
 $$L\stackrel{u}\longrightarrow M\stackrel{v}\longrightarrow N\stackrel{w}\longrightarrow\Omega^{-1}L$$ 
obtained  as above from a pushout diagram, where we denote again by $u,v,w$ the induced morphisms in the stable category.
\noindent So triangle closure in the stable category is induced by extension closure in the 
category of finitely generated graded $R$-modules. Note also, that, for every 
$R$-modules $M$ and $N$ we have:
$$\underline\Hom_R(\Omega M,N)\cong\Ext^1_R(M,N).$$
 \noindent Recall that a full subcategory $\mathcal T$ of $\underline\md  R$ 
is {\it thick}, if it is closed under triangles, direct summands and the suspension functor. 
We are interested in describing  the thick subcategories of $\mathcal D^b(\coh {\bf P}^n)$ 
or equivalently (see \cite{BGG}) of $\underline\md R$. In this paper we restrict to
thick subcategories consisting of modules of small complexity.
We will need some 
background on the stable graded category of $R$-modules and we first recall that $R$ is 
a symmetric Koszul algebra whose Koszul dual is the polynomial algebra $S$ in $n+1$ indeterminants.  Let $M$ be 
a finitely generated graded $R$-module and let 
$$(\mathbb{F})\hspace*{0.5cm} \cdots\rightarrow F^2 \rightarrow F^1 \rightarrow F^0 \rightarrow M \rightarrow 0$$
\noindent be a minimal graded free resolution of $M$. The module  $M$ is called a {\em linear module} 
if each free module $F^i$ in ($\mathbb{F}$) is generated in degree $i$.

\noindent  Considering this  minimal free resolution $\mathbb{F}$  of $M$, we let
the {\em i-th Betti number} $\beta_i(M)$ of $M$ be the rank of the free module $F_i$. 
Recall that the {\it complexity} of a finitely generated module $M$ over the exterior algebra 
is the infimum of the set 
$$\{d\in\mathbb N|\beta_i(M)\le ci^{d-1}\ \text{for some positive}  \ c\in\mathbb Q \  \text{and all} \  i\geq 0\}$$ 
\noindent So an $R$-module $M$ has {\it  complexity 1}, if its Betti numbers $\beta_i(M)$ are bounded.

\medskip
\noindent Note that for a finite dimensional self-injective algebra, the complexity of a module can also be
defined by using  $\dim_{\Bbbk} \Omega^iM$ instead of $\beta_i(M)$, and the module $M$
has complexity 0, if and only if it is projective. It is 
well-known that for every module  $M$ over the exterior algebra in $n+1$ indeterminates we have $\cx M\le n+1$. 
The complexity $n+1$ is attained; for instance the trivial module has this complexity. 

\medskip
\noindent A very useful characterization of the complexity of an $R(V)$-module
$M$, using {\em M-regular sequences} is shown in \cite[Section 3]{AAH}. It follows directly from this characterization, that the $R$-modules of complexity smaller than $\dm_{\Bbbk} V$ have even 
dimension, and that for a subspace $U\subseteq V$, the factor module $R/\langle U \rangle$ has complexity 
$\dm_{\Bbbk} U$.

\noindent In particular, if  $\dm_{\Bbbk} U = 1$, then $U$ is a point in the projective space  ${\bf P}(V)$ and  $M_U = R/\langle U\rangle$ is cyclic and has 
complexity 1. 

\medskip
\noindent We will  need the following result due to Eisenbud \cite{E}:
\begin{thm}\label{Eis} Let $M$ be a finitely generated  graded $R$-module of complexity 1. 
\item[(a)] If $M$ has no indecomposable projective direct summand, then, as $R$-modules, $M\cong \Omega M$, 
and as graded modules $M\cong\Omega M(1)$.
\item [(b)] The module $M$ has a filtration 
$0=M_0\subset M_1\subset \cdots \subset M_{s-1}\subset M_s =M$
with $M_i/M_{i-1} \cong  M_{\xi_i}(j_i)$ for some $\xi_i \in {\bf P}(V)$ and some $j_i\in \mathbb Z$.
 \qed
\end{thm}
\medskip

\noindent This result has no counterpart for modules with complexity bigger then one, as the following
example shows:

\begin{Example}
For $R=R(x,y,z)$ let $M$ be the 4-dimensional $R$-module with basis $\{ e_1,e_2,f_1,f_2\}$ and the
following operations of $R$ on $M$: $e_iz = f_i$, $e_1x=e_2y=f_ix=f_iy=f_iz=0$ and $e_1y=f_2$, $e_2x = f_1$.
One directly checks that $(z)$ is a maximal $M$-regular sequence, hence $ \cx M =2$ by \cite{AAH}.  Any proper
indecomposable submodule of $M$ has odd dimension, hence complexity 3. It is not hard to prove
that the module $M$ is linear.
\end{Example}

\medskip
\noindent We mention the following most probably well-known result, and we include a proof for the convenience of the reader:

\begin{lemma}  Let $U$ be a nonzero subspace of $V$.Then, the module $M=R/{\langle U\rangle}$ is linear.
\end{lemma}
\begin{proof} The result is clear if $U=V$ so assume that that $U$ is a proper subspace of $V$. Let $\{x_0,\ldots ,x_n\}$ be a basis of $V$. Without loss of generality we may assume that $U$ is spanned by $\{x_0,\ldots ,x_k\}$ for some $k<n$. Note that $R=R(x_0,\ldots ,x_k)\boxtimes R(x_{k+1},\ldots,x_n)$ where $\boxtimes$ denotes the skew tensor product (see \cite{PP,MVZ2}). Note also that we have an isomorphism of graded $R$-modules $R/{\langle U\rangle}\cong\Bbbk\boxtimes R(x_{k+1},\ldots,x_n)$ and, as skew tensor products of linear modules are linear, the proof is complete.   
\end{proof}

\medskip
\begin{Example}\label{complexity example}

\noindent   Let $R = R(x,y,z)$, and let $M$ be a four dimensional module with Loewy length two, generated in degree $0$, and with
basis $\{e,ex,ey,ez\}$. Then there are no $M$-regular elements, hence $\cx M = 3$, by \cite{AAH}.
If $M' = M/ M\langle\xi\rangle$,  where $\xi$ is a no zero linear form, then $\dm_{\Bbbk}M' = 3$ and $\cx M' = 3$.
\end{Example}
\medskip

\noindent A graded $R$ module $M$ is called {\it weakly Koszul} if its quadratic dual 
$M^!$ is a linear $S$-module (see \cite{GM,HI}), where $S$ denotes the 
$\Ext$-algebra $\Ext^*_R(R/{J} ,R/{J})$, with $J$ being the radical of $R$. Note that $S\cong\Bbbk [y_0,\ldots ,y_n]$, where $\{y_0,\ldots ,y_n\}$ is a dual basis of $\{x_0,\ldots ,x_n\}$.
The weakly Koszul modules 
can be described also in terms of certain types of extensions of shifts of linear modules (see \cite[p. 679]{MVZ1}). 
First, we recall the following definition:

\begin{dfn} An extension $0\rightarrow L\rightarrow M\rightarrow N\rightarrow 0$ of graded $R$-modules is called 
relative if for each $k\ge 0$ we have $MJ^k\cap L=LJ^k$. 
\end{dfn}

\noindent For instance, every short exact sequence $0\rightarrow L\rightarrow M\rightarrow N\rightarrow 0$ of linear 
$R$-modules is a relative extension. We have the following (see \cite{MVZ1}):

\begin{pro}\label{filtration by linear}
Let $M$ be a weakly Koszul $R$-module and assume that a minimal set of homogeneous 
generators of $M$ is distributed in degrees $i_0<i_1<\ldots <i_p$. Let $L=\langle M_{i_0}\rangle$ 
be the homogeneous submodule of $M$ generated by the degree $i_0$ part of $M$. Then $L$ is a 
graded shift of a linear module, $M/L$ is a weakly Koszul module generated in degrees $i_1<\ldots <i_p$, 
and the extension $0\rightarrow L\rightarrow M\rightarrow M/L\rightarrow 0$ is also a relative extension.
\end{pro}

\noindent Using the above proposition, one can use an inductive procedure to show that every weakly 
Koszul module can be filtered with linear modules (and their degree shifts) using relative extensions. The following proposition was proved in \cite{GM}.

\begin{pro}\label{syzygy behavior} Let $0\rightarrow L\rightarrow M\rightarrow N\rightarrow 0$ 
be a relative extension of weakly Koszul modules. 
Then for each $k\ge 0$ we have relative extensions of weakly Koszul modules 
$0\rightarrow\Omega^k L\rightarrow\Omega^k M\rightarrow\Omega^k N\rightarrow 0$.
\end{pro}

\noindent An immediate consequence is that if  $M$ is a weakly Koszul $R$-module generated in 
degrees $i_0<\cdots <i_p$, then $\Omega M$ is also weakly Koszul and is generated in 
degrees $i_0+1<\cdots <i_p+1$.  

\medskip
\noindent An easy inductive argument 
shows that the following result holds over any selfinjective Koszul algebra :

\begin{prop}\label{complexity middle} Let $0\rightarrow A\rightarrow B\rightarrow C\rightarrow 0$ be a relative 
extension of weakly Koszul modules. Then $\cx B=\max\{\cx A,\cx C\}$. \qed
\end{prop}

\noindent What makes the weakly Koszul modules relatively ubiquitous over the exterior algebra 
is the fact that given any finitely generated $R$-module $M$, then $\Omega^iM$ is weakly  
Koszul for some positive integer $i$ (see for instance \cite{HI,MVZ1}).  If $M$ is indecomposable with
$\cx M =1$, then by Eisenbud's theorem \ref{Eis} $\Omega^rM \cong M(-r)$. Thus Theorem \ref{Eis} 
also implies the following

\begin{cor}\label{cx1}All the indecomposable $R$-modules of complexity one are weakly Koszul. \qed
\end{cor}
\smallskip

\noindent We also have the following (\cite{MVZ1}):

\begin{prop}\label{background} Let $R$ be the exterior algebra in $n+1$ indeterminates and $M\in\md R$ be an 
indecomposable graded non projective module. Then: 
\begin{enumerate}
\item $\tau M=\Omega^2M(n+1)$, where $\tau$ denotes the \AR translate in the category of graded $R$-modules.
\item If $n=1$  and $M$ is weakly Koszul, then  $M$ is linear, or a graded shift of a linear module, 
and $M$ and $\tau M$ are both generated in the same degree.
\end{enumerate}
\end{prop}

\noindent We start with the following observation:

\begin{lemma} Let $M$ be an indecomposable finitely generated graded module over the exterior algebra 
in $n+1$ indeterminates, $n\ge 1$. Then: 
\begin{enumerate}
\item If $\cx M=1$, then for every $n\ge 1$, $\tau M=M(n-1)$. In particular, if $n=1$ then $\tau M=M$. 
\item Let $L\rightarrow M\rightarrow N\rightarrow\Omega^{-1}L$ be a distinguished triangle in 
$\underline\md R$. If $\cx L=\cx M=i$, then $\cx N\le i$. In particular, if $\cx L=\cx M=1$, 
then either $\cx N=1$ or $N$ is free.
\end{enumerate}
\end{lemma}
\begin{proof} We know by Eisenbud's theorem that $M\cong\Omega M(1)$, so also $M\cong\Omega^2M(2)$, 
and moreover $M$ must be weakly Koszul. As we saw earlier, $\tau M=\Omega^2M(n+1)$, so the first part 
follows immediately. For the second part it suffices to assume that none of the modules involved is 
free. But then we know that every distinguished triangle $L\rightarrow M\rightarrow N\rightarrow\Omega^{-1}L$ is 
isomorphic in the stable category of $R$ to a triangle obtained from a pushout diagram that yields a 
short exact sequence $0\rightarrow M\rightarrow N\oplus F\rightarrow\Omega^{-1}L\rightarrow 0$ for 
some free $R$-module $F$. This completes the proof of the lemma.
\end{proof}
\noindent Since indecomposable modules of complexity one are weakly Koszul, see Corollary \ref{cx1},
Proposition \ref{filtration by linear} has the following refinement
(see also \cite{GWW1}):

\begin{cor}\label{cx1-filtr} Let $M$ be an indecomposable module of complexity one over the exterior 
algebra and assume that $M$ has a minimal set of generators in degrees $i_0<i_1<\ldots<i_p$. 
Let $L=\langle M_{i_0}\rangle$ be the homogeneous submodule of $M$ generated by the degree $i_0$ 
part of $M$ and let $N=M/L$. Then both $L$ and $N$ are non projective and have complexity one, 
and therefore $M$ can be filtered by graded shifts of linear modules of complexity one.
\end{cor}
 \begin{proof}  It is obvious that neither $L$ nor $N$ are projective, so by 
\ref{filtration by linear} and \ref{syzygy behavior} we have an induced short exact sequence 
$0\rightarrow\Omega L\rightarrow\Omega M\rightarrow\Omega N\rightarrow 0$ and it is also 
clear that the submodule of $\Omega M$ generated in degree $i_0+1$ is isomorphic to $\Omega L$. 
So  both $L$ and $N$  have also complexity one by \ref{complexity middle}.
The corollary follows now from \ref{filtration by linear} and induction on the number of 
degrees of generation for $M$.
 \end{proof}

\noindent Note that the category of graded modules over the exterior algebra is closed 
under tensor products, where the tensor product is taken over the ground field, and is defined as follows: 
For two finitely generated graded $R$-modules $M$ and $N$, we define 
$(M\otimes N)_k=\bigoplus_i (M_i\otimes N_{k-i})$ where the $R$-module structure is induced 
by 
$$(m\otimes n)\xi=(-1)^i m\xi \otimes n + m\otimes n\xi $$
 for each $m\in M$, $n\in N_i$, 
and for each homogeneous element $\xi$ of degree one   in $R$. Moreover, if 
$0\rightarrow A\rightarrow B\rightarrow C\rightarrow 0$ is a short exact sequence of graded $R$-module, 
then for every graded $R$-module $M$ we have an induced sequence of graded $R$-modules 
$0\rightarrow M\otimes A\rightarrow M\otimes B\rightarrow M\otimes C\rightarrow 0$. 
Following \cite{BCR} we say that a thick subcategory $\mathcal T$ has {\it ideal closure}, 
if for each $M\in\mathcal T$ and $X\in\underline\md R$, we have that $M\otimes X\in\mathcal T$. 
Let $S$ denote the unique simple $R$-module concentrated in degree $0$, so that 
$S(-i)$ is the unique simple module concentrated in degree $i$. Then it is easy to see 
that $M\otimes S(-i) \cong M(-i)$. Let $\mathcal T_M$ denote the thick subcategory 
generated by $M$. We have the following:

\begin{prop} Let $R$ be the exterior algebra in $n+1$ indeterminates, and let $M$ be an indecomposable 
graded module of complexity 1. Then the thick subcategory $\mathcal T_M$ generated by $M$, 
contains the \AR components containing all the graded shifts of $M$. 
In particular, $\mathcal T_M$ has ideal closure. 
\end{prop}
\begin{proof} As we have seen above, if $M$ is a finitely generated indecomposable
graded module of complexity 1, then $\Omega^{-1}M\cong\ M(1)$ and $M(-1)\cong\Omega M$, 
so for each integer $i$, the graded shifts $M(i)\in\mathcal T_M$. 
Moreover, since for such a module $\tau M\cong M(n-1)$, it follows 
that  $\mathcal T_M$ contains also the \AR component containing $M$ and all its graded shifts. 
To show that $\mathcal T_M$ has ideal closure we proceed by induction on the graded length using
as a starting point the fact that if $L$ is a semisimple module generated in a single degree, 
then $M\otimes L$ is isomorphic to a direct sum of graded shifts of $M$. To be more 
specific let $L=L_{i_0}\oplus\cdots\oplus L_{i_0+t}$ be a graded $R$-module and 
consider the exact sequence $0\rightarrow L_{\ge i_0+1}\rightarrow L\rightarrow L_{ i_0}\rightarrow 0$.
We have an induced exact sequence of $R$-modules 
$0\rightarrow M\otimes L_{\ge i_0+1}\rightarrow M\otimes L\rightarrow M\otimes L_{ i_0}\rightarrow 0$. 
By induction, $M\otimes L_{ i_0}$ and $M\otimes L_{\ge i_0+1}$ are in 
$\mathcal T_M$ so $M\otimes L\in\mathcal T_M$ too.
\end{proof}

 \noindent Let $\mathcal T$ be a thick subcategory 
of $\underline\md R$ containing an indecomposable  module $M$ of complexity 1, where $R$ is the exterior algebra in $n+1$ indeterminates. 
For $n=1$ the module $M$ lies 
in a homogeneous tube, and if $n>1$ then $M$ lies in a ${\mathbb Z}A_\infty$ component by \cite[Theorem 5.7]{MVZ1}. 
So  we may always assume that $M$ is quasi-simple.

\section{Thick subcategories generated by modules of complexity one}

\noindent Let $R=R(V)$ denote the exterior algebra in $n+1$ indeterminates $x_0, x_1,\cdots, x_n$.  
Following \cite{E}, the complexity one modules of smallest length are cyclic of dimension $2^n$, 
linear and, up to degree shift, of the form $R/{\langle\xi\rangle}$ with $\xi \in {\bf P}(V)$, \\

\noindent Let $M=M_{\xi}$ be such a module. We describe first the thick subcategory $\mathcal T_{M}$ of 
$\underline\md R$ generated by $M$. It is clear that $\Omega M(1)=M$ and $\tau M=M(n-1)$. 
We start by collecting some facts about these modules. 

\begin{lemma}\label{Hom lemma} Let $\xi\neq\eta$ be two points in ${\bf P}(V)$. Then:
\begin{enumerate}
\item[(i)]  $\Hom_R(M_{\xi},M_{\eta})=0$ and
 $\Hom_R(M_{\xi},M_{\xi})=\underline\Hom_R(M_{\xi},M_{\xi})$ is one dimensional.
\item[(ii)]   $\underline\Hom_R(M_{\xi},M_{\eta}(i))=0$ for each integer $i$. 
Equivalently,  $\Ext_R^k(M_{\xi},M_{\eta}(i))=0$ for each integer $i$,  and for each non-negative integer $k$.
\item[(iii)] $\underline\Hom_R(M_{\xi},M_{\xi}(i))$ has dimension $\binom{n}{i}$ if $0\le i\le n$
and is zero otherwise.
\item[(iv)] $\Hom_R(M_{\xi}(j),M_{\eta})=0=\Hom_R(M_{\xi}(j),M_{\xi})$ for all $j>0$ and for all \\
$j<-n$.
\item[(v)]  $\Hom_R(\Omega^iM_{\xi},M_{\xi}) = \underline{\Hom}_R(\Omega^iM_{\xi},M_{\xi})$ for all $i\ge 0$.
\end{enumerate} 
\end{lemma}

\begin{proof}

\noindent (i) If there was a nonzero homomorphism $f\colon M_{\xi}\rightarrow M_{\eta}$, then $f$ would 
take the degree zero part of $M_{\xi}$ onto the degree zero part of $M_{\eta}$ 
hence $f$ would be onto and a dimension argument would imply that $f$ would have to
 be an isomorphism. Consequently $M_{\xi}$ and $M_{\eta}$ will have the same annihilator 
in $R$ contradicting the fact that $\xi\neq\eta$.\\
Since $M_{\xi}$ has simple top, the graded Hom space $\Hom_R(M_{\xi},M_{\xi})$ is one dimensional.
It is obvious 
that no nonzero degree zero homomorphism from $M_{\xi}$ to $M_{\xi}$ can factor through 
a free module so we obtain immediately that $\Hom_R(M_{\xi},M_{\xi})=\underline\Hom_R(M_{\xi},M_{\xi})$. 

\medskip
\noindent (ii) By making a change of basis if necessary we may assume that
 $R$ is 
the exterior algebra in $x_0, x_1,\cdots,  x_n$ with $x_0\in\xi$ and $x_n\in\eta$. The result is obvious if 
either $i\le 0$ or $i\ge n$, so assume that $0<i\le n-1$. Let $e_1$ span the degree 
zero part of $M_{\xi}$ and let $e_2$ span the degree $-i$ part of $M_{\eta}(i)$. 
A graded homomorphism $f\colon M_{\xi}\rightarrow M_{\eta}(i)$  is given by $f(e_1)=e_2a$
where $a$ is a homogeneous element of degree $i$ in $x_0, x_1,\cdots, x_{n-1}$, but 
since $e_1$ is annihilated by $x_0$ it turns out that $a=bx_0$ where $b$ is homogenous 
of degree $i-1$ in $x_1,\cdots, x_{n-1}$. 
Let $\epsilon\colon M_{\xi}\rightarrow R(1)$ 
be the inclusion of $M_{\xi}$ into its injective envelope. Then $\epsilon (e_1)= e_3x_0$
where $e_3$ spans the degree $-1$ part of $R(1)$.  We show that $f$ factors through $\epsilon$. 
A typical graded homomorphism $h\colon R(1)\rightarrow M_{\eta}(i)$ has the 
form $h(e_3)=e_2c$ where $c$ is a homogenous element of degree $i-1$ in the 
indeterminates $x_0, x_1,\cdots, x_{n-1}$. So $h\epsilon (e_1)= e_2c'$ where $c'$ is 
a homogenous element of degree $i$ containing $x_0$ as a factor. It is clear 
that $c$ can always be chosen so that $h\epsilon=f$ and the proof is complete.

\medskip

\noindent (iii) The result is trivial if $i$ is negative or $i>n$, so 
assume $0\le i\le n$. Again we may assume without loss of generality $x_0\in \xi$.
 Let $e_1$ span 
the degree zero part of $M_{\xi}$ and let $e_2$ span the degree $-i$ part of $M_{\xi}(i)$. 
A degree zero embedding $\epsilon$ from $M_{\xi}$ into its injective envelope $R(1)$ is 
given by $\epsilon(e_1)= e_3x_0$ where $e_3$ spans the degree $-1$ part  of $R(1)$, and a 
homomorphism $h\colon R(1)\rightarrow M_{\xi}(i)$ is given by $h(e_3)=e_2a$ where $a$ is 
a homogenous element of $R$ of degree $i-1$ in $x_1,\cdots, x_n$. Since $x_0$ 
annihilates $M_{\xi}$ and all its graded shifts we get that 
$\Hom_R(M_{\xi},M_{\xi}(i))=\underline\Hom_R(M_{\xi},M_{\xi}(i))$ and the rest follows immediately.

\medskip
\noindent (iv,v) These also follow immediately.
\end{proof}

\noindent As a consequence of part (i) of the above lemma we see that 
$\Ext^1_R(M_{\xi}(1),M_{\xi})$ is one dimensional since it is isomorphic 
to $\underline\Hom_R(M_{\xi},M_{\xi})$. In fact up to isomorphism we have
a nonsplit exact sequence $0\rightarrow M_{\xi}\rightarrow R(1)\rightarrow M_{\xi}(1)\rightarrow 0$ 
spanning this Ext space. Part (iii) of the lemma implies that $\Ext_R^1(M_{\xi},M_{\xi}(n-1))$ 
is one dimensional, and then up to isomorphism we have a nonsplit short exact sequence 
of graded $R$-modules $0\rightarrow M_{\xi}(n-1)\rightarrow X_{\xi}\rightarrow M_{\xi}\rightarrow 0$, which is an Auslander-Reiten sequence.  
It is easy to see that the module $X_{\xi}$ is indecomposable since by \cite{E} there are no indecomposable modules of dimension smaller than the dimension of $M$ in this Auslander-Reiten component.

\medskip
\noindent Since for all graded $R$-modules $X,Y$ and for each $t\ge 1$ we have isomorphisms 
$\Ext^t_R(X,Y) \cong \underline\Hom_R(\Omega^tX,Y) \cong \underline\Hom_R(X,\Omega^{-t}Y)$,
we get from part (iii) of Lemma \ref{Hom lemma}:
\begin{cor}\label{Ext cor} Let $R= R(V)$ and let $\xi \in {\bf P}(V)$. Then
\begin{enumerate}
\item[(i)] $\Ext^1_R(M_\xi,M_\xi(i))\ne 0$  if and only if $0\le i+1\le n$. 
\item[(ii)] $\Ext^1_R(M_{\xi},M_{\xi})$ is $n$-dimensional. 
$\Ext^2_R(M_\xi,M_\xi(i))\ne 0$  if and only if $0\le i+2\le n$.
\item[(iii)] $\Ext^1_R(M_\xi,M_\xi(i))\ne 0$ and $\Ext^2_R(M_\xi,M_\xi(i))=0$ if and only if $i=n-1$. \qed
\end{enumerate} 
\end{cor}

\medskip
\noindent   Putting together Theorem \ref{Eis} and part (ii) of \ref{Hom lemma}, we obtain:
\begin{cor} \label{oneM} Let $R = R(V)$ and let $X$ be an indecomposable finitely generated $R$-module of 
complexity one. Then there exists a unique $\xi \in {\bf P}(V)$, such that $X$ has a filtration
with subfactors $M_{\xi}(j_i)$, for suitable  $j_i\in \mathbb Z$. \qed
\end{cor}

\medskip
\noindent  Having this result in mind, Corollary \ref{cx1-filtr} then implies that for the study of
indecomposable modules of complexity one, we should consider linear modules $X$ of complexity
one, which have a filtration $0=X_0 \subset X_1 \subset \cdots \subset X_t = X$ with  
$X_i/X_{i-1} \cong M_{\xi}$ for a fixed $\xi \in {\bf P}(V)$. Let $M=M_{\xi}$ and $\cF (M) $ be the full subcategory of 
$\md R$
consisting of the modules having a filtration with subfactors isomorphic to $M$. Since $M$ 
has trivial endomorphism ring, by \cite{R1} the category $\cF (M)$ is a full exact and extension closed  abelian subcategory of  $\md R$
and $M$ is its unique (up to isomorphism) simple object. As in the classical case, every semi-simple object in $\mathcal F(M)$ is isomorphic to a direct sum of copies of $M$, any two different simple subobjects of a module in $\mathcal F(M)$ have zero intersection. Each object in $\cF (M)$ has a finite composition series,
hence $\cF (M)$ is a length category.

\medskip
\noindent Following Gabriel \cite[Sect.7]{G1} one should
then consider the Ext-quiver $\mathcal Q$ of this category, since it already gives a lot of informations
on $\mathcal F(M )$: 
 The quiver $\mathcal Q$ has one vertex and  $n = \dim_{\Bbbk}\Ext_R^1(M,M)$ loops.
 As usual we will identify the category of finite dimensional $\Bbbk$-linear representations
of $\mathcal Q$ with the category $\md\Bbbk\mathcal Q$.  Since $\mathcal Q$ has $n$ loops, 
the algebra $\Bbbk\mathcal Q \cong \Bbbk\langle t_1,\ldots ,t_n\rangle $-the free algebra in $t_1, \ldots , t_n$.
\smallskip

\noindent By Gabriel (\cite{G1}) any length-category is equivalent to
the category of finite dimensional modules over some pseudo-compact ring $A$ (see \cite{G2}). We will show that in our situation $A = \Bbbk[[t_1,\cdots,t_n]]$.

\smallskip
\noindent Obviously all the nonzero objects in $\cF (M)$ are linear of complexity 
one. 
Since $\mathcal F(M)$ is a length category, for  each nonzero object $X\in\mathcal F(M)$ we may consider its radical,
$\rad_{\mathcal F} X$- the intersection of all maximal subobjects of $X$ in $\mathcal F(M)$. Let us show that $\rad_{\mathcal F} X\in\mathcal F(M)$. Being an abelian category, $\mathcal F(M)$ is closed under finite intersections of subobjects. Let $(Y_j)_{j\in J}$ be a family of subobjects of some $Y\in \cF (M)$. Since $Y$ has finite $\cF (M)$-length, there exists a finite subfamily
$\{Y_1,\ldots ,Y_t\}$ such that $C =\cap_{1\leq i \leq t}Y_i \subset Y_j$ for all $j\in J$ that is, $C = \cap_{j\in J}Y_j$, so $\cF (M)$ is closed under arbitrary intersections as well.
\medskip

\noindent  Letting
$\rad_{\mathcal F}^0X=X$ and $\rad_{\mathcal F}^{i+1}X = \rad_{\mathcal F}(\rad_{\mathcal F}^iX)$, we obtain the ``Loewy series"
$$X\supset\rad_{\mathcal F} X\supset\rad_{\mathcal F}^2X\supset\cdots\supset\rad_{\mathcal F} ^{r-1}X\supset\rad_{\mathcal F} ^rX =0.$$
The smallest $r$ with the property that $\rad_{\mathcal F} ^rX=0$ is called the
{\em Loewy length} of $X$, written $\ell\ell(X)$. The factor $X/{\rad_{\mathcal F} X}$ is called the
$\cF$-{\em top} of $X$. If  $\cF$-$\tp X$ is simple, then $X$ is  called $\cF$-local or simply local.
All the factors $\rad_{\mathcal F}^iX/{\rad_{\mathcal F} ^{i+1}X}$ are semisimple
in the category $\mathcal F(M)$. The largest semi-simple subobject of $X$ in $\cF (M)$ is called 
the $\cF$-socle of $X$.

\smallskip
\noindent Our first aim is to give a precise description of this category $\cF (M)$. 

\noindent For each $i\in\mathbb N$ define the full subcategories $\mathcal F^{(i)} =
\{X\in\mathcal F(M)|\ell\ell(X)\leq i\}.$  We have an ascending chain of subcategories
$$0=\mathcal F^{(0)}\subset \mathcal F^{(1)} \subset \mathcal F^{(2)}\subset\cdots,$$
\noindent and $\mathcal F(M)= \bigcup_{i\in\mathbb N}\mathcal F^{(i)}$.  The inclusion functors $\mathcal F^{(i)}\to\mathcal F^{(i+1)}$ are right adjoint to the truncation functors
$\mathcal F^{(i+1)}\to\mathcal F^{(i)}$, defined by  $Y\mapsto Y/{\rad_{\mathcal F} ^{i}}Y$.
The categories $\mathcal F^{(i)}$
are exact abelian subcategories of $\mathcal F(M)$, but they are not closed under extensions.
 For each $i$, let $\mathcal L^{(i)}$
 be the set of local objects in $\mathcal F^{(i)}$ whose tops equal the unique 
 simple object $M$. We claim that the length of each object in $\mathcal L^{(i)}$
is bounded by $(n^i-1)/(n-1)$.

\smallskip 
\noindent The claim is clear for $i=1$.  For $i>1$ and $X\in \mathcal L^{(i)}$, then as $\dm_{\Bbbk}\Ext_R^1 (M,M) =n$,
we have that $\rad_{\cF}X/\rad^2_{\cF}X \cong M^r$, with $r\leq n$, so $\rad_{\cF}X$ is an epimorphic image of a sum
of $r$ local modules $L_1,\ldots , L_r$, all contained in $\rad_{\cF}X$.
By induction, since for each $j$, $\ell (L_j) \leq (n^{i-1}-1)/(n-1)$, we get $$\ell(X)\leq n(n^{i-1}-1)/(n-1)+1=(n^i-n)/(n-1)+1=(n^i-1)/(n-1)$$
which proves the claim. By \cite{K}, there exists in $\mathcal L^{(i)}$ a unique
object $P^{(i)}$ of maximal length with the properties that each object
in $\mathcal L^{(i)}$ is an epimorphic image of $P^{(i)}$,
and that for any $X \in\mathcal F^{(i)}$, each
epimorphism $X\to P^{(i)}$ splits. It is also easy to check that the object $P^{(i)}$ is projective in $\cF^{(i)}$ in the usual sense, meaning that, for each
epimorphism $X\to Y$ in $\mathcal F^{(i)}$, the induced map $\Hom_R (P^{(i)},X)\to\Hom_R (P^{(i)},Y)$ is
surjective.  
\noindent
The object $P^{(i)}$ is then the minimal projective generator in the category $\cF^{(i)}$,
therefore by Morita theory the functor $\Hom_R(P^{(i)},-)$ induces an equivalence 
$\mathcal F^{(i)}\cong\md\End_R( P^{(i)})$.

\begin{lemma} For all $i\in\mathbb{N}$ we have
 $P^{(i)}/\rad_{\mathcal F}^{\ i-1}P^{(i)} \cong P^{(i-1)}$.
\end{lemma}
\begin{proof} Since $P^{(i)}$ generates all elements in
$\mathcal L^{(i)}$, it generates also all
the elements in $\mathcal L^{(i-1)}$. So every object in $\mathcal L^{(i-1)}$
is a quotient of $P^{(i)}/{\rad_{\mathcal F}^{i-1}P^{(i)}}$, and, since $P^{(i-1)}$ has maximal composition
length in $\mathcal L^{(i-1)}$, the result follows.
\end{proof}

\noindent This lemma will enable us, to construct the objects $P^{(i)}$ and their endomorphism rings
inductively. 
For $i=1$ the category $\mathcal F^{(1)}=\add{M}$ is semisimple,
so $P^{(1)} = M$, hence
$\End_RM = {\Bbbk}$. 
For the inductive step we use universal short exact sequences as follows:
Let $X\in \mathcal F^{(i)}$ be indecomposable with $\dm_{\Bbbk}\Ext_R^1(X,M)=a$.  Let $(\eta_{s})_{1\leq s\leq a}$ be
a basis of $\Ext_R^1(X,M)$, say $$\eta_{s}: 0\to M \to E_{s}\to X\to 0.$$ 
We have a pullback diagram
$$ 
\xymatrix
{ &0\ar[r]&M^{a}\ar@{=}[d]
\ar[r] &E\ar[d]\ar[r]
&X\ar^{\Delta}[d] \ar[r] &0\\ &0 \ar[r]&M^{a}\ar[r]
&\oplus_{s}E_{s}\ar[r] &X^{ a}\ar[r]& 0} 
$$ 
where $\Delta : X\to X^{ a}$ is the diagonal embedding. The top short exact sequence $0\to M^{a}\to E\to X \to 0$ is called the {\em universal short exact sequence} (\cite{Bo}) for $\Ext_R^1(X,\add M )$ and the connecting homomorphism 
$\Hom_R(M^{a},M)\to \Ext_R^1(X,M)$ is an isomorphism.
Note that $E$ is uniquely determined up to isomorphism.

\begin{Remark}
 \noindent   For each $k>1$, the $\cF$-top of $\rad_{\cF}P^{(k)}$ equals $M^n$. Let $e_i\in\rad_{\cF}P^{(k)}$ be such that $e_i+\rad^2_{\cF}P^{(k)}$ is a generator of the $i$-th direct  summand in $M^n$ above and let $e$ be the generator of the 
 $\cF$-top of $P^{(k)}$. For each $1\le i\le n$, let $\pi_i\colon P^{(k)}\to e_iR$ be the $\cF$-projective cover such that $\pi_i(e)=e_i$, and let $\epsilon_i$ denote the inclusions $e_iR\to P^{(k)}$. We will view each $t_i$ as the endomorphism $\epsilon_i\pi_i\colon P^{(k)}\rightarrow P^{(k)}$. Let now $g\colon P^{(k)}\to P^{(k)}$ be an endomorphism. Since $g(\rad_{\cF}P^{(k)})\subseteq\rad_{\cF}P^{(k)}$, $g$ induces a homomorphism
 $$\overline g\colon\cF{\mbox{\rm -top}}P^{(k)}\to\cF{\mbox{\rm -top}}P^{(k)}.$$ 
 If $\overline g\neq 0$, then $g$ is an automorphism, as $P^{(k)}$ is a local object in $\mathcal F(M)$. Otherwise, the image of $g$ must be contained in $\rad_{\cF}P^{(k)}=\sum_{i=1}^nt_i(P^{(k)})$.
\end{Remark}

\begin{Example}\label{P^2}  We give an explicit construction of $P^{(2)}$ that will be used shortly. For each $1\le s\le n$, consider the extensions $$\eta_s\colon 0\to M\to E_s\to M \to 0$$ where $M=M_{\xi}$ and $x_0\in\xi$. Then, as $R(x_1,\cdots,x_n)=R/{\langle x_0\rangle}$-modules, the modules $E_s$ and $M^2$ are isomorphic, and are each generated over $R$ by two generators: $e$, and $e_s$, satisfying $ex_0 = e_sx_s$, and   $e_sx_0 =0$. It is clear that the set $\{\eta_s\}_{1\le s\le n}$ forms a $\Bbbk$-basis for $\Ext^1_R(M,M)$. We now construct the universal
short exact sequence for $\Ext_R^1(M,\add M )$, and we get  $$0\to M^n \to P^{(2)}\to M \to 0$$ where $E=P^{(2)}$
has a minimal generating system $\{e,e_1,\ldots ,e_n\}$ in degree 0, subject to the relations
$ex_0 = \sum_s e_sx_s$ and $e_sx_0 =0$ for all $1\leq s\leq n$.
We have seen in the previous remark that we may consider each $t_i$ as being an endomorphism of $P^{(2)}$, taking $P^{(2)}$ inside $\rad_{\cF}P^{(2)}$. Since  $\rad^2_{\cF}P^{(2)}=0$, we have that $t_j(e) = e_j$ and $t_i(e_j) =0$ for all $i$ and $j$. Then, since we are considering only the degree zero homomorphisms, it follows that $t_1,\ldots ,t_n$ are linearly independent
and form a basis of the vector space of endomorphisms factoring through $\rad_{\cF}P^{(2)}$. Clearly $t_it_j =0$ for all $i$ and $j$, and 
this means that $\End_RP^{(2)}\cong \Bbbk [t_1,\ldots t_n]/\langle t_1,\ldots t_n\rangle ^2$ and  is commutative.
\end{Example}

 \noindent  Consider the universal short exact sequence $0\to\rad_{\cF}^{d}P^{(d+1)}\to P^{(d+1)}\to P^{(d)}\to 0$
 and let $f\in\End_RP^{(d+1)}$.  Since $\rad_{\cF}^{d}P^{(d+1)} $ is an $f$-stable submodule, $f$ induces for all $d>0$ 
 endomorphisms $f_d\in \End_RP^{(d)}$, hence algebra homomorphisms  $\Phi_d\colon\End_RP^{(d+1)}\to\End_RP^{(d)}$, for all
 $d>0$.
 
\begin{lemma}\label{square} The homomorphisms  $\Phi_d\colon\End_RP^{(d+1)}\to\End_RP^{(d)}$ are surjective and
the diagrams
$$ 
\xymatrix
{ &\mathcal F^{(d)}\ar^{\cong}[d]\ar@{^{(}->}[r]
&\mathcal F^{(d+1)}\ar^{\cong}[d]
\\ &\md\End_RP^{(d)}\ar@{^{(}->}[r]
&\md\End_RP^{(d+1)}}
$$
commute for all the natural numbers $d$.
\end{lemma}
\begin{proof}  Since $P^{(d)}\in \cF^{(d+1)}$ and $P^{(d+1)}$ is projective in $ \cF^{(d+1)}$, each endomorphism of
$P^{(d)}$ can be lifted to an endomorphism  of $P^{(d+1)}$, and this proves the surjectivity of the $\Phi_d$s.
The equivalences $\cF^{(d)}\to\md\End_RP^{(d)}$ for $d\geq 1$ are given by the functors
$\Hom_R(P^{(d)},-)$. But $P^{(d+1)}/\rad^d_{\cF}P^{(d+1)}\cong P^{(d)}$, and
$\rad^d_{\cF}P^{(d+1)}$ is in the kernel of every homomorphism $f\colon P^{(d+1)}\to U$ for $U\in \cF^{(d)}$. This implies that
 $\Hom_R(P^{(d+1)},-)_{|\mathcal F^{(d)}}=\Hom_R(P^{(d)},-)_{|\mathcal F^{(d)}}$.
\end{proof}

\noindent  We now can formulate the conclusion of these constructions:
\begin{thm}\label{nilmod} Let $R=R(x_0\ldots ,x_n)$ for some positive integer $n$, and let $M=M_{\xi}$ be a local $R$-module of complexity one and dimension $2^n$. Then the category $\mathcal F(M)$ is 
equivalent to the category of finite dimensional modules over the power series ring $\Bbbk[[t_1,\ldots ,t_n]]$.  
\end{thm}
\begin{proof} We will show that  $\End_RP^{(d)}\cong\Bbbk [t_1,\ldots ,t_n]/{\langle t_1,\ldots ,t_n\rangle^d}.$ Then, the claim will follow by Lemma \ref{square} by taking limits.
To prove this result, a detailed study of the modules $P^{(d)}$ is necessary. We assume that $x_0\in\xi$ and
we will show by induction on $d$ that:

\medskip
\noindent (A)  $\rad_{\cF}^{d-1}P^{(d)}$ is $\mathcal F$-semisimple, and is a direct sum of $\binom{n+d-2}{d-1}$ copies of $M$.  Note that $\binom{n+d-2}{d-1}$ is the number of pairwise different monomials of degree $d-1$ in
$\Bbbk [t_1,\ldots ,t_n]$.
The generators of these copies of $M$ are all in degree zero. Denoting these generators by $e_{i_1\ldots i_{d-1}}$
with $1\leq i_j \leq n$, we have for all permutations $\pi\in S_{d-1}$:
 $$e_{i_1\ldots i_{d-1}}=  e_{i_{\pi(1)}\ldots i_{\pi(d-1)}}$$
 \noindent where we consider the indices ${i_1\ldots i_{d-1}}$ as ``monomials" in the ``letters" $1,\ldots ,n$.

 \medskip
 \noindent (B)  For each generating element $e_{i_1\ldots i_{d-2}}+\rad_{\cF}^{d-1}P^{(d)}$ of $\rad_{\cF}^{d-2}P^{(d)}/{\rad_{\cF}^{d-1}P^{(d)}}$ 
 we have
$$e_{i_1\ldots i_{d-2}}x_0=\sum_{r=1}^ne_{i_1\ldots i_{d-2}r}x_r.$$

\noindent Statement (A) is trivially true for $P^{(1)} =M$, and both (A) and (B) were proved in example \ref{P^2} for $d=2$.
Assume that (A) and (B) hold true for all $P^{(s)}$ with $s\leq d$. Therefore we may choose a minimal
system of generators of the $\cF$-Loewy layers of $P^{(d)}\colon e, e_i, e_{i,j},\ldots e_{i_1\ldots i_{d-1}}$ with $e\in\cF$-top of $P^{(d)}$,
and with $e_{i_1\ldots i_s}$ in the $\cF$-top of $\rad_{\cF}^sP^{(d)}$, satisfying (A) and (B) for all $s\leq d-1$.

\smallskip
\noindent We construct $P^{(d+1)}$ as middle term of the universal short exact sequence
$$0\to M^c\to P^{(d+1)}\to P^{(d)}\to 0,$$ where $c=\dm_{\Bbbk}\Ext_R^1(P^{(d)},M)=\dm_{\Bbbk}\underline\Hom_R(P^{(d)},M(1))$. As an $R(x_1\ldots x_n)$-module, $M$ is free since it is indecomposable of maximal Loewy length. Thus, $P^{(d)}$
is equal to a direct sum of $\sum_{s=0}^{d-1}\binom{n+s-1}{s}$ copies of $M$ and is also a free $R(x_1\ldots x_n)$-module.
But $P^{(d)}$ is indecomposable as $R$-module (by the action of $x_0$) since it is  is a local object in the 
full length category $\mathcal F(M)$. In order to evaluate $c$, let us describe all the graded homomorphisms from $P^{(d)}$ to $M(1)$, as well as those factoring through a free module. 
 Let $g\colon P^{(d)}\to M(1)$ be an $R$-linear map.  Then, $g$ takes each of 
the generators $e, e_i, \ldots e_{i_1\ldots i_{d-1}}$ of (the linear module) $P^{(d)}$ into an element of degree $0$ in $M(1)$.

\smallskip
\noindent The cyclic module $M(1)$ is generated degree $-1$, and let us denote by $f$ its generator. Then,  as a $\Bbbk$-vector space, $M(1)_0 = \oplus_{i=1}^n \Bbbk f x_i$. Hence:
$$\begin{array}{ccc}
g(e)                      & =       &  f\sum_{r=1}^n a_rx_r \\
g(e_i)                   & =       &   f\sum_{r=1}^n a^i_rx_r \\
g(e_{ij})                 & =       &   f\sum_{r=1}^n a^{ij}_rx_r \\
\vdots     &\vdots              &\vdots                                  \\
g(e_{i_1\ldots i_s}) & =       &   f\sum_{r=1}^na^{i_1\ldots i_s}_rx_r \\
\vdots     &\vdots              &\vdots       \\
\end{array}$$
\noindent By induction, $e_{i_1\ldots i_{s}} =  e_{i_{\pi(1)}\ldots i_{\pi(s)}}$ for all $1\leq s\leq d-1$, so we
also get that $a^{i_1\ldots i_s}_r = a^{i_{\pi (1)}\ldots i_{\pi (s)}}_r$. In this way, for every choice of the scalars
$a_r, a_r^i, \ldots a^{i_1\ldots i_s}_r$ , the map $g$ is a well-defined
$R(x_1\ldots x_n)$-linear map,  since $P^{(d)}$ is free over $R(x_1\ldots x_n)$.  We want to define $g$ as an $R$-homomorphism, so the action of $x_0$ has to be considered separately: Since $M(1)$ is annihilated by $x_0$,
we have $g(px_0) =0$ for all $p\in P^{(d)}$.
This implies that $$0 = g(ex_0) = \sum_{i=1}^n g(e_i)x_i=f\sum_{i=1}^n\sum_{r=1}^na^i_rx_r x_i.$$
Since $x_rx_i+x_ix_r=0$ for all $r$ and $i$, we get $a^i_r = a^r_i$ for all $i$ and $r$. This means that the
set of admissible $n^2$-tuples $(a^i_r)_{1\leq i,r\leq n}$ with  $a^i_r = a^r_i$ form a subspace of $\Bbbk^{n^2}$ of dimension
 $\binom{n+1}{2}$ .
 \smallskip
 
\noindent Consider now for $1<s< d-1$  the element $e_{i_1\ldots i_s}$. We get
$$0= g(e_{i_1\ldots i_s}x_0)= \sum_{j=1}^ng(e_{i_1\ldots i_sj}x_j) = f\sum_{j=1}^n\sum_{r=1}^n a^{i_1\ldots i_sj}_rx_r x_j.$$
Again, we have $a^{i_1\ldots i_sj}_r=a^{i_1\ldots i_sr}_j$ for all $r$ and $j$. Additionally, we
see that $a^{i_1\ldots i_si_{s+1}}_r=a^{i_{\pi(1)}\ldots i_{\pi(s)}i_{\pi(s+1)}}_r$ for all $\pi\in S_{s+1}$. Consequently we have 
 $a_{i_{s+2}}^{i_1\ldots i_si_{s+1}}=a^{i_{\pi(1)}\ldots i_{\pi(s)}i_{\pi(s+1)}}_{ i_{\pi (s+2)}}$ for all $\pi\in S_{s+2}$.
Hence  the families of coefficents   $(a_{i_{s+2}}^{i_1\ldots i_si_{s+1}})$ with these relations form $\Bbbk$-vector
spaces of dimensions $\binom{n+s}{s+1}$.
\smallskip

\noindent Each admissible choice of these coefficients  $a_i,a_i^j,\ldots  ,a_{i_{s+2}}^{i_1\ldots i_si_{s+1}}$
defines a unique  $R$-linear map $g\colon P^{(d)}\to M(1)$.
Therefore, the $\Bbbk$-dimension of $\Hom_R(P^{(d)},M(1))$ is $\sum_{j=1}^{d}\binom{n+j-1}{j}$. 

\smallskip
\noindent  Consider now the vector space $\mathcal P_R(P^{(d)},M(1))$ of maps $P^{(d)}\to M(1)$ factoring through
a projective $R$-module, hence factoring through $R(1)$-the projective cover of $M(1)$. 
Let $\rho\colon P^{(d)}\to R(1)$ and $\pi\colon R(1)\to M(1)$. Denote the generator of $R(1)$ by $\hat f$, so $\pi (\hat f) =f$. Since $\pi\rho\in\mathcal P_R(P^{(d)},M(1))$, we get for $\rho$, similar to the preceding procedure
$$\begin{array}{ccc}
\rho(e)                      & =       &  \hat f(\sum_{r=1}^n a_rx_r  + ax_0)\\
\rho(e_i)                   & =       &    \hat f(\sum_{r=1}^n a^i_rx_r + a^ix_0) \\
\rho(e_{ij})                 & =       &    \hat f(\sum_{r=1}^n a^{ij}_rx_r  + a^{ij}x_0)\\
\vdots     &\vdots              &\vdots                                  \\
\rho(e_{i_1\ldots i_s}) & =       &    \hat f(\sum_{r=1}^n a^{i_1\ldots i_s}_rx_r + a^{i_1\ldots i_s}x_0) \\
\vdots     &\vdots              &\vdots       \\
\end{array}$$
The coefficients $a^{i_1\ldots i_s}_r$  satisfy the same relations as in the first part of the proof
 ($a_i^j = a_j^i$, $\ldots$),
since $\pi\rho\in\mathcal P_R(P^{(d)},M(1))\subseteq\Hom_R(P^{(d)},M(1))$. In this way we have defined an $R(x_1\ldots x_n)$-linear
map. Since it has to be $R$-linear, the action of $x_0$ has to be considered as well.
We have $$\rho(e)x_0=\hat f(\sum_{r=1}^n a_rx_r  + ax_0)x_0  = \hat f\sum_{r=1}^n a_rx_r x_0$$
since $x_0^2=0$. On the other hand, we have:
$$\rho (ex_0)=\rho(\sum_j e_jx_j)= \sum_j\hat f(\sum_{r=1}^na^j_rx_r + a^jx_0)x_j=\hat f(\sum_{j,r}a^j_rx_rx_j+a^jx_0x_j).$$ 
Since $a_i^j = a_j^i$ we have $\sum_{j,r}a^j_rx_rx_j =0$
and since we want $\rho(ex_0)=\rho(e)x_0$, we see that $a_j = -a^j$, for $1\leq j \leq n$.  Similarly, we get for $s<d-1$
\begin{align*}
\rho(e_{i_1\ldots i_s})x_0\notag&=\hat f(\sum_{1\le r\le n}a^{i_1\ldots i_s}_rx_r + a^{i_1\ldots i_s}x_0)x_0\notag\\&=
\hat f(\sum_{1\le r\le n}a^{i_1\ldots i_s}_rx_r x_0)=\rho(e_{i_1\ldots i_s}x_0)=\rho( \sum_j e_{i_1\ldots i_{s+1,j}}x_j)\notag\\& =  
\sum_j\hat f(\sum_{1\le r\le n}a^{i_1\ldots i_{s+1,j}}_rx_r + a^{i_1\ldots i_{s+1,j}}x_0)x_j=\hat f\sum_j a^{i_1\ldots i_{s+1,j}}x_0x_j\notag,
\end{align*}
since $\sum_j\sum_{r}a^{i_1\ldots i_{s+1,j}}_rx_r x_j =0$. This means that $a^{i_1\ldots i_s}_r= -a^{i_1\ldots i_{s+1,r}}$. 

\noindent Finally, for $d-1$ we get $e_{i_1\ldots i_{d-1}}x_0 =0$ and therefore $$0=\rho (e_{i_1\ldots i_{d-1}}x_0)=\hat f\sum_{r=1}^na^{i_1\ldots i_{d-1}}_rx_r x_0.$$ 
\noindent Thus, all the coefficients
$a^{i_1\ldots i_{d-1}}_r$ are zero. Notice that the map $\pi\rho$ is independent  of the choice of 
the coefficient $a\in\Bbbk$, and the coefficients  $a^{i_1\ldots i_s}$ are determined by the $a^{i_1\ldots i_{s-1}}_{i_s}$,
hence $\dm_{\Bbbk}\mathcal P_R(P^{(d)},M(1))=\sum_{j=1}^{d-1}\binom{n+j-1}{j}$. Consequently
\begin{align*}
\dm_{\Bbbk}\Ext_R^1(P^{(d)},M)\notag&=\dm_{\Bbbk}\underline{\Hom}_R((P^{(d)},M(1))\notag\\&
=\dm_{\Bbbk}\Hom_R(P^{(d)},M(1))-\dm_{\Bbbk}\mathcal P_R(P^{(d)},M(1))\notag\\&=\binom{n+d-1}{d}.
\end{align*}
Therefore $\rad_{\cF}^{d}P^{(d+1)}$  is a direct sum
of $\binom{n+d-1}{d}$ copies of $M$.
\medskip

\noindent  Next, we construct generators $e_{i_1\ldots i_d}$ for $\rad_{\cF}^{d}P^{(d+1)}$ 
satisfying the conditions (A) and (B). For this, we consider a projective cover $P$ of $\rad_{\cF}^{d-1}P^{(d+1)}$ 
in $\cF^{(2)}$.
 Since  $\rad_{\cF}^{d-1}P^{(d)}$ is a direct sum of  $\binom{n+d-2}{d-1}$ copies of $M$,  the
projective cover of  $\rad_{\cF}^{d-1}P^{(d+1)}$ in $\cF^{(2)}$ is a direct sum off $\binom{n+d-2}{d-1}$ copies
of $P^{(2)}$. The $\cF$-top of $\rad_{\cF}^{d-1}P^{(d+1)}$ is generated by the elements
$ e_{i_1\ldots i_{d-1}} +\rad_{\cF}^{d}P^{(d+1)}$. Denote the generators of the 
 $\cF$-tops of the copies of $P^{(2)}$ in the projective cover $P$ by $f_{i_1\ldots i_{d-1}}$, 
 with the convention that $f_{i_1\ldots i_{d-1}}=f_{i_{\pi (1)}\ldots i_{\pi (d-1)}}$, for all $\pi\in S_{d-1}$.
 Let us set $f_{i_1\ldots i_{d-1}}x_0=\sum_jf_{i_1\ldots i_{d-1}j}x_j$. Then the elements $f_{i_1\ldots i_{d-1}j }$
generate the $\cF$-radical of $P$.
\noindent Let 
$\gamma\colon P\to\rad_{\cF}^{d-1}P^{(d+1)}$ be the projective cover
with $\gamma(f_{i_1\ldots i_{d-1}})=e_{i_1\ldots i_{d-1}}$, and define 
$\gamma(f_{i_1\ldots i_{d-1}j})=e'_{i_1\ldots i_{d-1}j}$.  As an immediate consequence, we get that 
$e_{i_1\ldots i_{d-1}}x_0 = \sum_je'_{i_1\ldots i_{d-1}j}x_j$. 

\smallskip
\noindent {The identities
 $0=-e_{i_1\ldots i_{d-2}}x_0^2=\sum_re_{i_1\ldots i_{d-2}r}x_0x_r=\sum_r\sum_se'_{i_1\ldots i_{d-2}rs}x_sx_r$ imply
that $e'_{i_1\ldots i_{d-2}rs} = e'_{i_1\ldots i_{d-2}sr}$.} But $e_{i_1\ldots i_{d-2}i_{d-1}} =
e_{i_{\pi(1)}\ldots i_{\pi(d-2)}i_{\pi(d-1)}}$ for $\pi\in S_{d-1}$. Therefore the elements
$e_{i_1\ldots i_{d-2}i_{d-1}i_d} = e'_{i_1\ldots i_{d-2}i_{d-1}i_d}$ satisfy the required properties from (A) and (B).
\medskip

\noindent  Finally we have to show that $\End_RP^{(d)}=\Bbbk [t_1,\ldots ,t_n]/\langle t_1,\ldots ,t_n\rangle^d$.
It is trivial for $d=1$ and by \ref{P^2}, we know that 
$\End_RP^{(2)}=\Bbbk [t_1,\ldots ,t_n]/\langle t_1,\ldots ,t_n\rangle^2$ with the action $t_j(e) = e_j$ and
$t_i(e_j)=0$. 
The functor $$\Hom_R(P^{(d)},-)\colon\cF^{(d)}\longrightarrow\md\End_RP^{(d)}$$ is an equivalence
and the image of $P^{(d)}$ under this functor is the local algebra $\End_RP^{(d)}$ with identity $1_d$. Let $Y$ be module over $\End_RP^{(d)}$. Then each $R$-linear map $g\colon\End_RP^{(d)}\to Y$ is uniquely determined by the value $g(1_d)$, and for every $y\in Y$ there exists a unique $\End_RP^{(d)}$-homomorphism $g_y\colon\End_RP^{(d)}\to Y$ with
$g_y(1_d)=y$. This means that for each object $X$  in the subcategory $\cF^{(d)}$, and for every
scalar $x\in X_0$, there exists a unique homomorphism $f_x\colon P^{(d)}\to X$ with $f_x(e)=x$ and every $R$-homomorphism
$f\colon P^{(d)} \to X$ must have this form. Since $P^{(d)}_0$ is the $\Bbbk$-span of the set
$\{e, e_i, e_{ij},\ldots e_{i_1\ldots i_{d-1}}\}$, there exists exactly one endomorphism $t_i\in\End_RP^{(d)}$
such that $t_i(e) = e_i$. 
\smallskip

\noindent As $t_i$ is $R$-linear, $t_i(ex_0) = e_ix_0=\sum_j e_{ij}x_j$ and
$t_i(ex_0) = \sum_jt_i(e_jx_j)$, which means $t_i(e_j)=e_{ij}=t_it_j(e)$. By $(A)$, $e_{ij}=e_{ji}$, so
$t_it_j = t_jt_i$ and the $\binom{n+1}{2}$ different products are linearly independent.
We assume by induction that for all $s<d-1$ we have $e_{i_1\ldots i_s} = t_{i_1}\cdots t_{i_s} (e)$,
hence $t_{i_1} \cdots t_{i_s} = t_{i_{\pi (1)}}\cdots t_{i_{\pi (s)}} $ for all $\pi\in S_s$. All the
$\binom{n+s-1}{s}$ monomials in $t_1,\ldots ,t_n$ of degree $s$ are linearly independent in $\End_RP^{(d)}$.
Consider  $t_i(e_{i_1\ldots i_{d-3}}x_0)= e_{i_1\ldots i_{d-3}i}x_0= \sum_je_{i_1\ldots  i_{d-3}ij}x_j
 =\sum_j t_i(e_{i_1\ldots i_{d-3}j})x_j$.
 Consequently we get $$t_i(t_{i_1}\cdots t_{i_{d-3}}t_j(e))=t_i(e_{i_1\ldots i_{d-3}j})= e_{i_1\ldots i_{d-3}ji},$$
 and so, $e_{i_1\ldots i_{d-1}} = e_{i_{\pi  (1)}\ldots i_{\pi (d-1)}}$ implies that $t_{i_1}\cdots t_{d-1}=t_{i_{\pi (1)}}\cdots t_{i_{\pi (d-1)}}$ for all permutations $\pi\in S_{d-1}$.
Therefore, the $\binom{n+d-2}{d-1}$ monomials
in $t_1,\ldots t_n$ of degree $d-1$ are linearly independent in $\End_RP^{(d)}$. It is easy to show that 
$t_i(e_{i_1\ldots i_{d-1}}) =0$ for all $i$ and all $(i_1\ldots i_{d-1})$. This means that the algebra
$\End_RP^{(d)}$ is isomorphic to $\Bbbk [t_1,\ldots ,t_n] /\langle t_1,\ldots ,t_n\rangle^d$.
\end{proof}

\noindent Let $M=M_{\xi}$ and ${\mathcal F}(\{M(i)\}_{i\in\mathbb Z})$ denote the full 
subcategory of $\md R$ consisting of those graded modules 
having a filtration with composition factors of the type $M(i)$, $i\in\mathbb Z$. Note that the category ${\mathcal F}(\{M(i)\}_{i\in\mathbb Z})$ is not an abelian category, if $n>1$.
\medskip

\noindent One consequence of  Theorem \ref{Eis} and Corollaries \ref{cx1-filtr} and \ref{oneM},
 every indecomposable module of dimension larger 
than $2^n$ in ${\mathcal F}(\{M(i)\}_{i\in\mathbb Z})$ is a relative extension of two modules 
in the same subcategory, hence this subcategory is closed under syzygies and by duality, 
under cosyzygies. Finally, denote by $\underline{\mathcal F}$ the image of $\mathcal F$ in 
$\underline\md R$. 
 We have the following consequence of this discussion:
 
\begin{prop} $\mathcal T_M=\underline{\mathcal F}(\{M(i)\}_{i\in\mathbb Z})$. 
\end{prop}
\begin{proof} We only need to show that $\underline{\mathcal F}(\{M(i)\}_{i\in\mathbb Z})$ is closed 
under triangles in $\underline\md R$. But every triangle 
$A\rightarrow B\rightarrow C\rightarrow\Omega^{-1}A$ with $A$ and $C$ in 
$\underline{\mathcal F}({\{M(i)\}_{i\in\mathbb Z})}$ comes from a short 
exact sequence $0\rightarrow A\rightarrow I(A)\oplus B\rightarrow C\oplus I\rightarrow 0$ 
where $I(A)$ is the injective envelope of $A$, and $I$ is some injective $R$-module. As ${\mathcal F}(\{M(i)\}_{i\in\mathbb Z})$ 
is closed under extensions, $\underline{\mathcal F}(\{M(i)\}_{i\in\mathbb Z})$ is  
closed under triangles, and so it is a thick subcategory of the stable module 
category containing $M$.
\end{proof}

 \noindent In order to prove that the thick category generated by a module of the form
$M_{\xi}$ does not contain a proper non zero thick subcategory, we need to recall a few facts. We may assume without loss of generality that $R=R(\xi,x_1,\cdots x_n)$, and thus the polynomial algebra $S=\Bbbk[z,y_1,\cdots,y_n]$, where $\{z,y_1,\cdots,y_n\}$ is the dual basis of $\{\xi,x_1,\cdots x_n\}$. Consider the point $S$-module $\Bbbk[z]=S/{\langle y_1,\cdots,y_n\rangle}$. Under the BGG correspondence 
$$\underline\md R\rightarrow\dfrac{D^b(\md R)}{\mathcal P}\rightarrow\dfrac{D^b(\md S)}{\mathcal {FD}} $$

\noindent where $\mathcal P$ denotes the thick subcategory of perfect complexes in $D^b(\md R)$ and $\mathcal {FD}$ is the thick subcategory of $D^b(\md S)$ consisting of complexes of finite dimensional modules, the module $M_{\xi}$ corresponds to the shifted stalk complex $\Bbbk[z][-n]$ in $D^b(\md S)$.  In the following we use some arguments, which are well known to topologists, see e.g. \cite{DS,N,T} and we are grateful  to Greg Stevenson for very helpful conversations. We will  need the notion of support of a triangulated subcategory of $D^b(\mathcal S)$. 
If $C$ is a complex in $D^b(\mathcal S)$, we define its support $$\Supp(C)=\{\frak p\in\Spec^hS|\ \ H^*(C)_{\frak p}\ne 0\},$$ where $\Spec^hS$ denotes the homogeneous spectrum of $S$ consisting of the relevant homogeneous prime ideals in $S$. Then, the support of a triangulated subcategory $\mathcal A$ of $D^b(\mathcal S)$, is the union of the supports of the objects in $\mathcal A$. 
 \begin{thm} Let  $R=R(V)$ be the exterior algebra in $n+1$ indeterminates and let $\xi\in{\bf P}(V)$. Let 
$M=M_{\xi}$ be a cyclic $R$-module of complexity one. Then each nonzero element of $\mathcal T_M$ generates $\mathcal T_M$. 
 \end{thm}
 \begin{proof}  
 We need to show that $M\in\mathcal T_X$ for each nonzero 
$X\in\mathcal T_M$. Let $\sigma(\mathcal T_M)$ and $\sigma(\mathcal T_X)$ denote the images of $\mathcal T_M$ and of $\mathcal T_X$ in $D^b(\md S)$ under the BGG correspondence. Both are thick subcategories closed under tensor products since this was the case with $\mathcal T_M$ and $\mathcal T_X$. But the support of $\sigma(\mathcal T_M)$ in $\Spec^hS$, consists only of the ideal $\langle y_1,\cdots,y_n\rangle$ of height $n$. Therefore, $\Supp(\sigma(\mathcal T_M))=\Supp(\sigma(\mathcal T_X))$ and by \cite[Theorem 3.15]{T} (see also \cite{DS}), we have $\mathcal T_M=\mathcal T_X$.
\end{proof}

 \noindent  Summarizing the results of this section, we have:
 
 \begin{thm}\label{modules of cx 1} Let $R=R(V)$ be the exterior algebra in $n+1$ indeterminates and let $\mathcal T$ be the thick subcategory of $\underline\md R$ 
consisting of all the modules of complexity one. Then 
$\mathcal T=\bigsqcup_{\xi\in {\bf P}(V)}\mathcal T_{M_{\xi}}$ is the decomposition 
of $\mathcal T$ into its irreducible thick subcategories. \qed
 \end{thm}

\section{Modules with self-extensions}

\noindent  It is an open problem whether all the indecomposable modules of non-maximal complexity over the exterior algebra have graded self-extension. We cannot give a complete answer even in the complexity one case. We prove in this section that in the case $n=2$, every module of complexity one or two, has self-extensions, while in the higher dimensional case, we prove this result for linear modules having non maximal complexity. We will use an argument which seems to be well known to geometers, and we are grateful to Greg Stevenson for explaining us how it can be deduced from
B\"ohning's result  \cite[2.1.4]{B} : A sheaf without self-extensions over   ${\bf P}(V)$ is locally free.
We start with a  preliminary result. We have the following (\cite[3.8]{AAH}):

\begin{pro} Let $R=R(x_0,\cdots,x_n)$ and let $S=\Bbbk[x_0,\cdots,x_n]$ where the base field $\Bbbk$ is algebraically closed. Let $M$ be a finitely generated graded $R$-module. Then $\cx M$ equals the Krull dimension of the graded $S$-module $\mathcal E(M)=\Ext_R^*(M,\Bbbk)$. 
\end{pro}

\noindent This implies that if $\mathcal E(M)$ is torsion free, then $M$ must have maximal complexity. In particular, if the sheafification (\cite{S}) of $\mathcal E(M)$ is locally free, then the complexity of $M$ must equal $n+1$. 

\begin{thm} Let $R=R(x_0,\cdots,x_n)$ 
and let $M$ be an indecomposable linear module of complexity less or equal to $n$. Then $\Ext^1_R(M,M)\ne 0$.
\end{thm}
\begin{proof}  Assume that $\Ext_R^1(M,M)=0$. By Koszul duality, we have that $\Ext_S^1(X,X)=0$ where 
$X=\mathcal E(M)=\Ext_R^*(M,\Bbbk)$. Let $\widetilde X$ be the sheafification of $X$. It follows from \cite{HZ} that the 
sheaf $\widetilde X$ has no self extension either, and from the proof of \cite[2.1.4]{B} that $\widetilde X$ is locally free yielding a contradiction to our preceding remarks.
\end{proof}

\begin{Remark} It would be interesting to have also a proof of the above result that does not require a geometric argument. We have such a proof only in the case that the complexity of our $R$-linear module equals 1. This proof is presented in an earlier version of this paper, posted in arXiv: 1701.01149. More precisely, we prove that an indecomposable linear 
module of complexity one over the exterior algebra in $n+1$ indeterminates has self-extensions.
\end{Remark}

\noindent  Since we know the structure of all indecomposable modules of complexity one, see Theorem \ref{Eis}
and Corollary \ref{oneM}, we can
use Mukai's lemma (\cite[1.4]{Ru}) we  get a slight generalization:

\begin{prop} \label {selfex}Let $R$ be the exterior algebra in $n+1$ indeterminates
and let $X$ be an indecomposable  module of complexity one minimally generated in degrees $i_0<i_1<\ldots 
<i_{p-1}<i_p$ and assume either $i_1 = i_0+(n-1)$ or $ i_p = i_{p-1}+(n-1)$. 
Then $\Ext^1_R(X,X)\ne 0$.
\end{prop}
\begin{proof}  Let $i_1 = i_0+(n-1)$ and let $L= \langle X_{i_0}\rangle$ be the submodule of $X$ generated by $X_{i_0}$.Then $L$ is a degree shift
of a linear module of complexity one, hence $\Ext_R^1(L,L) \neq 0$ by the previous proposition.
Consider the short exact sequence $0\to L \to X \to X/L \to 0$. Then $\Hom_R(L,X/L)=0$ and
$\Ext_R^2(X/L,L)=0$, hence by \cite[1.4]{Ru} again, $\dm_{\Bbbk}\Ext_R^1 (X,X)\geq\dm_{\Bbbk} \Ext_R^1(L,L)>0$.
In the second case, $ i_p = i_{p-1}+(n-1)$ we use the dual argument.
\end{proof}

\begin{cor}  Ler $R$ be the exterior algebra in $n+1$ indeterminates and let $\mathcal C$ be an Auslander-Reiten
component in the Auslander-Reiten quiver of $\md R$ containing a shift of a linear module of complexity one.
Then all indecomposable modules in $\mathcal C$ have self-extensions.
\end{cor}
\begin{proof} Since $\tau X \cong X(n-1)$ for complexity one modules, any indecomposable module $Y$ in $\mathcal C$
has a minimal set of generators in degrees $i_0, i_0+(n-1), \ldots, i_0 + s(n-1)$ for some integer
$i_0$ and some $s\geq 0$, hence the above proposition applies.
\end{proof}

 \noindent  Using \cite[1.4]{Ru} we get in the two dimensional case:
 
\begin{prop} Let $R$ be the exterior algebra in $3$ indeterminates
and let $X$ be an indecomposable  module of complexity one or two. Then $\Ext^1_R(X,X)\ne 0$.
\end{prop}
\begin{proof}  If $X$ has complexity one, then $X$ is minimally generated in degrees $i_0, i_0+1, i_0+2, \cdots, i_0+r$,
hence it has self extensions by Proposition \ref{selfex}.
Let $X$ be an indecomposable module such that $\cx X = 2$ and 
and assume that $\Ext^1_R(X,X)=0$. By taking syzygies, we may assume without loss of generality that $X $ is weakly Koszul (\cite{MVZ1}). Again without loss of generality we may assume that $X$ is minimally generated in degrees $0,1,\ldots,p$. Then, as in the proof of \ref{selfex} there exists a linear $R$-module $U$ and a relative extension $0\to U\to X\to V\to 0$ of weakly Koszul modules of complexity  two. Moreover, $V$ is minimally generated in degrees $1,\ldots,p$, and we have $\Hom_R(U,V)=\Ext_R^2(V,U)=0$. From Mukai's lemma, we infer that $\Ext_R^1(U,U)=0$. Under Koszul duality and sheafification, $U$ corresponds to a rigid sheaf on the projective plane, so by \cite{D}, it must be a vector bundle. But the modules corresponding to vector bundles must have maximal complexity by \cite {AAH}, hence $\cx U=\cx X=3$ and we get a contradiction.
\end{proof}


\appendix 
  
 \section{The projective line case}
  \label{app:prelim}

 The thick subcategories of $\underline\md R$ where $R=R(x_0,x_1)$ is the 
exterior algebra in two indeterminates were described by Krause in \cite{Kr} using non-crossing partitions. Using the BGG correspondence this yields a description of the thick 
subcategories of $\mathcal D^b(\coh {\bf P}^1)$. In this appendix we will use an ``Auslander-Reiten theoretical approach" to formulate and recapture these results. Note that in the dimension one case, every 
indecomposable non projective graded $R$-module that is not simple has Loewy length 
(therefore graded length) equal to two. \\

\noindent If $M$ is indecomposable with complexity 1, then $\tau M\cong M$, hence $M$ lies in a homogeneous tube $\Sigma$,
the whole tube is contained in the thick subcategory $\mathcal T_M$ and any two indecomposable modules in $\Sigma$
generate the same thick subcategory. Also for every non zero linear form $\xi$, the modules $M_{\xi}$ defined at the beginning of this paper, are uniserial of length 
two, have complexity one, and are the graded quasi-simple $R$-modules. Therefore we may assume that the module $M$ is quasi-simple in $\Sigma$.

\begin{prop} Let $M$ be an indecomposable quasi-simple graded module of complexity 1. 
Let $N$ be another quasi-simple module of complexity 1 belonging to $\mathcal T_M$. 
Then $N\cong\Omega^{-i}M\cong M(i)$ for some $i\in\mathbb Z$.

\begin{proof} Using \ref{background} we infer that the module $M$ is linear, 
so it is generated in degree 0. We know that $\mathcal T_M$ contains the tube containing $M$,
 as well as all the tubes containing each $\Omega^iM=M(-i)$. 
  Since $\mathcal T_M$ is 
connected, there are nonzero maps between the homogeneous tubes contained in $\mathcal T_M$. 
So for every tube in $\mathcal T_M$ containing a quasi-simple module $L$, there is another tube  
$\mathcal T_M$ containing a quasi-simple module $N$ such that we must have 
$\underline\Hom_R(L,N)\ne 0$ or $\underline\Hom_R(N,L)\ne 0$. Since $L$ is weakly Koszul it follows 
from \ref{background} that it must be in fact, a (degree) shift of some linear module.
Assume that $\underline\Hom_R(L,N)\ne 0$. By the preceding remarks and Lemma 1.2.,
$N$ is either generated in degree 0, or -1. If $N$ is generated in degree 0, since both $L$ and 
$N$ are uniserial of graded length 2, every nonzero homomorphism from $L$ to $N$ must be an isomorphism,
so assume that $N$ is generated in degree
  $-1$. But $N$ being also of complexity 1, means that we may write $N=\Omega^{-1}K$ for some linear 
module $K$ of complexity 1, so we have $\Ext_R^1(L, K)=\underline\Hom_R(L,\Omega^{-1}K)\ne 0$.  
Thus we have a non split extension of linear modules 
$0\rightarrow K\rightarrow X\rightarrow L\rightarrow 0$ and an induced commutative diagram:
$$ 
\xymatrix
{ &0\ar[r]&K \ar^{f}[d]
\ar[r] &X\ar[d]\ar[r]
&L\ar@{=}[d] \ar[r] &0\\ &0 \ar[r] &L \ar[r]
&Y\ar[r] &L\ar[r]& 0} 
$$
where the bottom sequence is the \AR sequence ending at $L$. As before, the homomorphism $f$ is either 
an isomorphism or zero. But $f$ cannot be zero since this would imply a splitting of the \AR
sequence ending at $L$. So $f$ is an isomorphism, and then $N\cong\Omega^{-1}L$. The case 
$\underline\Hom_R(N,L)\ne 0$ is treated in a similar fashion. 
\end{proof}
\end{prop} 

\noindent As an immediate consequence we obtain a description of the thick subcategories 
generated by indecomposable modules of complexity 1:

\begin{prop} Let $M$ be an indecomposable $R$-module of complexity 1, and let 
$\mathcal T_M$ be the thick subcategory generated by $M$. Then $\mathcal T_M$ is 
the additive subcategory of $\underline\md^ZR$ generated by $\bigcup_{i\in\mathbb Z}\Sigma_i$, 
where for each integer $i$, $\Sigma_i$ denotes the graded \AR component containing the module $M(i)$.
\end{prop}
\begin{proof} { We have seen that $\mathcal T_M$ contains the $\add\bigcup_{i\in\mathbb Z}\Sigma_i$, 
so we only need to prove that the additive closure of this union is a triangulated 
subcategory of $\underline\md^{\mathbb Z}R$. Let $X\rightarrow Y\rightarrow Z\rightarrow\Omega^{-1}X$ 
be a triangle where $X,Z$ are in $\add\bigcup_{i\in\mathbb Z}\Sigma_i$. Obviously $Z$ is also isomorphic 
to a direct sum of modules of complexity 1. Let $C$ be an 
indecomposable summand of $Y$. We may assume that one of $\underline\Hom_R(X,C)$ and 
$\underline\Hom_R(C,Z)$ is nonzero.  Say we have a non-zero graded map from an indecomposable 
direct summand $X'$ of $X$  contained in some tube
 $\Sigma_j$ to  $C$. As in the proof of the previous lemma, we may assume that both $X'$ and $C$ are quasi-simple and it 
follows that $C\in\Sigma_j$ or $C\in\Sigma_{j-1}$.}
\end{proof}

\begin{lemma} Let $M$ be an indecomposable graded module of complexity one, 
and let $\mathcal T$ be a thick subcategory of $\underline\md R$ 
containing $M$ and a graded simple module. Then $\mathcal T=\underline\md R$.
\end{lemma}
\begin{proof} Since $\mathcal T_M$ contains all the graded shifts of $M$, 
we may assume without loss of generality that $M$ is quasi-simple generated 
in degree 0. We have an embedding $S(-1)\rightarrow M$
 and a homomorphism $M\rightarrow S$, where $S$ is the unique simple module 
concentrated in degree zero. From the distinguished triangle $$S(-1)\rightarrow
M\rightarrow S\rightarrow\Omega^{-1}S(-1)$$ 
\noindent we infer that $S(-1)\in\mathcal T$ and 
in this way we get that for each $i\in\mathbb Z$, $S(i)\in\mathcal T$. Let $N$ 
be an indecomposable non projective $R$-module of graded length two.
 We have a short exact sequence $0\rightarrow S(l)^t\rightarrow N\rightarrow S(l+1)^s\rightarrow 0$ 
for some positive integers $s,t$ and for an integer $l$. It follows that $N\in\mathcal T$.
\end{proof}

\noindent We obtain the following consequence:
\begin{cor} Let $M$ be an indecomposable graded module of complexity one, and let 
$\mathcal T$ be a thick subcategory of $\underline\md R$ containing 
$M$ and an indecomposable module belonging to one of the graded transjective components.
 Then $\mathcal T=\underline\md R$.
\end{cor}
\begin{proof} In view of the previous lemma, it suffices to show that $\mathcal T$ 
contains a graded simple module. Let $X\in\mathcal T$ be an indecomposable module of complexity two, 
and let $\Gamma$ be the transjective component containing it, 
$$\xymatrix@R=8pt@C=16pt{
 && F_1 \ar@<.5ex>[dr]\ar@<-0.5ex>[dr]  
& &
F_{-1} \ar@<.5ex>[dr]\ar@<-0.5ex>[dr]  
&&\\  \dots &
F_2\ar@<.5ex>[ur]\ar@<-0.5ex>[ur]  
&&
S \ar@<.5ex>[ur]\ar@<-0.5ex>[ur]  
&&
F_{-2} &\dots} $$ 

\noindent where $S$ is the graded simple module lying in $\Gamma$. Assume that $X=F_i$ for some $i\ge 1$. 
If $i=2k$ is even, $F_i=\tau^kS=\Omega^{2k}S(2k)$, 
so the graded simple module $S(2k)\in\mathcal T$. If $i=2k+1$, then $F_i=\tau^kF_1=\Omega^{2k}F_1(2k)$, 
implying that $F_1(2k)\in\mathcal T$. 
 But  $M(2k)$ is also in $\mathcal T$ and, as in proofs of the previous lemmas, we have a nonzero homomorphism 
 $f\colon M(2k)\rightarrow F_1(2k)$, or a non-zero homomorphism $f\colon F_1(2k)\rightarrow M(2k)$. 
In the first case, since both $F_1(2k)$ and $M(2k)$ are 
generated in the same degree and have linear resolutions, 
$\underline\Hom_R(M(2k),F_1(2k))=\Hom_R(M(2k),F_1(2k))$. But $M(2k)$ is quasi-simple, 
so $f$ is one-to-one and its cokernel is simple. So in this case too,
$\mathcal T$ contains a graded simple module. The remaining case, and the case when $X$ is 
of the form $F_{-i}$ for some $i\ge 0$ is treated the same way. 
\end{proof}


\noindent We turn our attention to the thick subcategories generated by the modules of complexity 2. 
Let $\Gamma_0$ be the transjective component 
containing the simple module $S$ generated in degree zero. 
Then for each integer $i$ we have $$F_{i+1}=\Omega F_i(1)=\Omega^{i+1}S(i+1)$$ 
\noindent where $F_0$ denotes the simple module $S$. Let $X\in\Gamma_0$. Then $\mathcal T_X$ contains 
all the syzygies and cosyzygies of $X$, and it follows that for each integer $j$, 
$\mathcal T_X$ contains a graded shift of $F_j$, so it must contain some graded simple module. 
So when describing $\mathcal T_X$ we may assume without loss of generality that $X$ is a graded 
simple module. We have the following:

\begin{lemma} Let $S$ be the graded simple module generated in degree zero. 
Then $\mathcal T_S$ is the additive closure of $\{F_i(-i)|i\in\mathbb Z\}$.
\end{lemma}

\begin{proof} $\add\{F_i(-i)|i\in\mathbb Z\}$ is obviously closed under syzygies and 
cosyzygies so to show that it is triangulated we 
need to show that if two terms of a distinguished triangle are in $\add\{F_i(-i)|i\in\mathbb Z\}$, 
then so is the third. {Let $X$ be an indecomposable module of complexity two. Without loss of generality we may clearly 
assume that $X=F_i=\Omega^iS(i)$ and belongs to the Auslander-Reiten component $\Gamma_0$ containing the 
unique simple module concentrated in degree zero.  Then $$\Ext^1_R(F_i(-i),F_i(-i))=\Ext^1_R(S,S)=0.$$ 
\noindent It is also easy to show that for all integers $i,j$:
$$\underline\Hom_R(F_i(-i),F_j(-j))=
\begin{cases}
     0 & \text{if}\ \ i\ne j. \\
     \Bbbk & \text{otherwise}.
\end{cases}$$
This clearly implies that each  $\add\{F_i(-i)|i\in\mathbb Z\}$ is a triangulated subcategory 
and the proof of the lemma is complete.}
\end{proof}

\begin{lemma} Let $\mathcal T$ be a thick subcategory of  $\underline\md R$, containing indecomposable
 modules $F_i(j)$ and $F_r(s)$ with $i+j\neq r+s$. Then $\mathcal T =\underline\md R$.
\end{lemma}
\begin{proof} Assume that $r+s>i+j$. Since for each $i,j$ we have $F_i(j)=\Omega^iS(i+j)$,  
we infer  from $F_i(j) \in \mathcal T$ that $S(i+j) \in \mathcal T$.  
Similarly $F_{r+s-(i+j)}(i+j)=\Omega^{r+s}S(i+j)$ is in $\mathcal T$. 
But $\underline\Hom_R(F_{r+s-(i+j)}(i+j), S(i+j))\neq 0$. As in the Kronecker 
algebra case there exists a short exact sequence
$$ 0\to M(i+j)\to F_{r+s-(i+j)}(i+j) \to S(i+j) \to 0,$$
\noindent where $M$ is an indecomposable module  of complexity 1. From this 
short exact sequence we get in  $\underline\md^Z_R$ a
triangle 
$$M(i+j) \to F_{r+s-(i+j)}(i+j) \to Si+j) \to \Omega^{-1}M(i+j)$$ 
\noindent Since $F_{r+s-(i+j)}(i+j)$ and $S(i+j)$ are in $\mathcal T$, $M(i+j)\in\mathcal T$ too. Corollary 2.5 then 
implies $\mathcal T=\underline\md R$.
\end{proof}
\noindent We summarize this discussion with the following description 
of the lattice of thick subcategories of $\underline\md R$:

\begin{thm} Let $R=R(x_0,x_1)$ be the exterior algebra in two indeterminates.
\
 \begin{enumerate}
\item The proper thick subcategories of $\underline\md R$ 
are of the form $\mathcal T_{\alpha}$, where $\phi\ne {\alpha\subseteq {\bf P}^1}$,
 and $\add\{F_i(-i+j)$ for all $i,j\in\mathbb Z\}$.
 \item The proper thick subcategories of $\underline\md R$ closed under tensor products  
are only of the form $\mathcal T_{\alpha}$, where $\phi\ne {\alpha\subseteq {\bf P}^1}$. \qed
\end{enumerate}
\end{thm}
\noindent The following picture is useful in illustrating the irreducible thick subcategories consisting of modules of complexity 2.

\adjustbox{scale=.9,center}{\begin{tikzcd}
 &&& \; \ar[d, dotted, no head] & \; & \; & \;\\
 \ar[r, dotted] & F_2(2) \ar[r, shift left] \ar[r, shift right] & F_1(2) \ar[r, shift left] \ar[r, shift right] & S(2)  \ar[r, shift left] \ar[r, shift right] & F_{-1}(2) \ar[ru, dashed, no head, red] \ar[r, shift left] \ar[r, shift right] & F_{-2}(2) \ar[ru, dashed, no head, red] \ar[r, dotted] & \; \\
 \ar[r, dotted] & F_2(1) \ar[r, shift left] \ar[r, shift right] & F_1(1) \ar[r, shift left] \ar[r, shift right] & S(1) \ar[ru, dashed, no head, red] \ar[r, shift left] \ar[r, shift right] & F_{-1}(1) \ar[ru, dashed, no head, red] \ar[r, shift left] \ar[r, shift right] & F_{-2}(1) \ar[ru, dashed, no head, red] \ar[r, dotted] & \; \\
 \ar[r, dotted] & F_2 \ar[r, shift left] \ar[r, shift right] & F_1 \ar[ru, dashed, no head, red] \ar[r, shift left] \ar[r, shift right] & S \ar[ru, dashed, no head, red] \ar[r, shift left] \ar[r, shift right] & F_{-1} \ar[ru, dashed, no head, red] \ar[r, shift left] \ar[r, shift right] & F_{-2} \ar[ru, dashed, no head, red] \ar[r, dotted] & \; \\
 \ar[r, dotted] & F_2(-1) \ar[ru, dashed, no head, red] \ar[r, shift left] \ar[r, shift right] & F_1(-1) \ar[ru, dashed, no head, red] \ar[r, shift left] \ar[r, shift right] & S(-1) \ar[ru, dashed, no head, red] \ar[d, dotted, no head] \ar[r, shift left] \ar[r, shift right] & F_{-1}(-1) \ar[ru, dashed, no head, red] \ar[r, shift left] \ar[r, shift right] & F_{-2}(-1) \ar[r, dotted] & \; \\
\; \ar[ru, dashed, no head, red] & \; \ar[ru, dashed, no head, red] & \; \ar[ru, dashed, no head, red] & \; \ar[ru, dashed, no head, red] & \; && \;
\end{tikzcd}}

 
\end{document}